\documentclass[a4paper]{amsart}
\usepackage{graphicx}
\usepackage{amssymb} 
\usepackage{stmaryrd}
\usepackage{dsfont}
\usepackage[mathcal]{euscript}
\usepackage{tikz}
\usepackage{tikz-cd}
\tikzset{labl/.style={anchor=south, rotate=270, inner sep=.5mm}}
\usepackage{hyperref}
\hypersetup{%
  bookmarksnumbered=true,%
  colorlinks=true,%
  linkcolor=blue,%
  citecolor=blue,%
  filecolor=blue,%
  menucolor=blue,%
  urlcolor=blue,%
  bookmarksopen=true,%
  bookmarksdepth=2,%
  pageanchor=true}

\title{Abelian versus triangulated  quotients}

\author{Henning Krause}
\address{Fakult\"at f\"ur Mathematik\\
Universit\"at Bielefeld\\ D-33501 Bielefeld\\ Germany}
\email{hkrause@math.uni-bielefeld.de}

\theoremstyle{plain}
\newtheorem{thm}{Theorem}
\newtheorem{prop}[thm]{Proposition}
\newtheorem{lem}[thm]{Lemma}

\theoremstyle{definition}

\theoremstyle{remark}

\numberwithin{equation}{thm}

\newcommand{\Abfree}{\operatorname{Ab}}
\newcommand{\add}{\operatorname{add}}

\newcommand{\Ext}{\operatorname{Ext}}

\newcommand{\Hom}{\operatorname{Hom}}
\newcommand{\HOM}{\operatorname{\mathcal{H}\!\!\;\mathit{om}}}

\renewcommand{\mod}{\operatorname{mod}}

\newcommand{\Ab}{\mathrm{Ab}} 
\newcommand{\ac}{\mathrm{ac}} 

\newcommand{\coh}{\mathrm{coh}}
\newcommand{\ex}{\mathrm{ex}}

\newcommand{\op}{\mathrm{op}}

\newcommand{\iso}{\xrightarrow{\raisebox{-.4ex}[0ex][0ex]{$\scriptstyle{\sim}$}}}

\newcommand{\longiso}{\xrightarrow{\ \raisebox{-.4ex}[0ex][0ex]{$\scriptstyle{\sim}$}\ }}

\newcommand{\lto}{\longrightarrow}

\newcommand*{\intref}[2]{\def\tmp{#1}\ifx\tmp\empty\hyperref[#2]{\ref*{#2}}\else\hyperref[#2]{#1~\ref*{#2}}\fi}

\def\A{\mathcal A} 
\def\B{\mathcal B} 
\def\C{\mathcal C}
\def\D{\mathcal D}

\def\calS{\mathcal S} 
\def\T{\mathcal T}

\def\bfD{\mathbf D} 
\def\bfK{\mathbf K}

\def\bbF{\mathbb F}

\def\bbZ{\mathbb Z}

\newcommand{\frU}{\mathfrak{U}}
\newcommand{\frV}{\mathfrak{V}}

\def\s{\sigma}
\def\t{\tau}

\def\La{\Lambda}
\def\Si{\Sigma}

\begin{document}



\begin{abstract}
  It is shown that any localisation of triangulated categories induces
  (up to an equivalence) a localisation of abelian categories when one
  passes to their abelianisations. From this one obtains for any
  enlargement of Grothendieck universes an example of an abelian
  category and a Serre subcategory within the smaller universe such
  that the corresponding quotient does only exist within the bigger
  universe. The second part of this note provides an analogue for the
  abelian hull of an arbitrary category.
\end{abstract}

\date{\today}

\maketitle

\section{Introduction}

Localisations of categories naturally arise in the study of abelian
and triangulated categories \cite{Ga1962,Ve1997}, but the construction
of a localisation raises some set-theoretic issues. More specifically,
given a category $\C$ and a class of morphisms $\Si$ in $\C$, there is
a universal functor $\C\to\C[\Si^{-1}]$ that inverts all morphisms in
$\Si$; see \cite{GZ1967}. For the localised category $\C[\Si^{-1}]$ it
is a priori not clear that for any given pair of objects the morphisms
between them form a set (and not a proper class).  Fortunately there
are explicit criteria in the literature which ensure that there is no
reason to worry, but often this issue is ignored.

For an abelian category $\A$ and a Serre subcategory $\C$ the \emph{quotient
category} $\A/\C$ is by definition the localisation $\A[\Si^{-1}]$ with
respect to the class $\Si$ of morphisms with kernel and cokernel both
in $\C$; see \cite[III.1]{Ga1962}. In this note we present examples
of such a quotient such that for some pair of objects the morphisms
between them do not form a set. For a concise exposition it is
convenient to use the setting of Grothendieck universes.

\section{Abelian quotients}

We fix a Grothendieck universe $\frU$ which is sufficiently big to
contain the set of integers; see \cite[I.1]{Ga1962}. If not mentioned
otherwise, we assume throughout that all categories are
\emph{$\frU$-categories}, which means that for any pair of objects
$X,Y$ the set $\Hom(X,Y)$ belongs to $\frU$. Given a pair of
categories $\C,\D$ we write $\HOM(\C,\D)$ for the category of functors
$\C\to\D$, with $\Hom(F,G)$ given by the set of natural
transformations $F\to G$ for any pair of functors $F,G$. This need not
be a $\frU$-category.

\begin{center}
  \textasteriskcentered \qquad \textasteriskcentered \qquad \textasteriskcentered
\end{center}\smallskip

For a pair $\A,\B$ of abelian categories let $\HOM_\ex(\A,\B)$ denote
the full subcategory of $\HOM(\A,\B)$ consisting of the exact functors
$\A\to\B$. Given a Serre subcategory $\C\subseteq\A$, the quotient
functor $\A\to\A/\C$ can be characterised as follows.

\begin{lem}\label{le:ab}{\cite[III.1]{Ga1962}}
  For an exact functor $F\colon\A \to\A'$ between abelian categories
  and a Serre subcategory $\C\subseteq\A$, the following are
  equivalent.
  \begin{enumerate}
    \item
The functor $F$ induces for any abelian category $\B$ a fully faithful functor
\[\HOM_\ex(\A',\B)\lto\HOM_\ex(\A,\B)\]
with essential image consisting of all functors $\A\to\B$ that
annihilate $\C$.
\item The functor $F$   induces an equivalence $\A/\C\iso\A'$.\qed
\end{enumerate}  
\end{lem}

\section{The abelianisation of a triangulated category}

Let $\T$ be a triangulated category.  Given any abelian category $\A$
we write $\HOM_\coh(\T,\A)$ for the full subcategory of $\HOM(\T,\A)$
consisting of the cohomological functors $\T\to\A$. Recall from
\cite[II.3]{Ve1997} that the \emph{abelianisation} $\T\to\Abfree(\T)$  is a fully
faithful cohomological functor that induces for any abelian category
$\A$ an equivalence
\[\HOM_\ex(\Abfree(\T),\A)\longiso\HOM_\coh(\T,\A).\]

For a triangulated subcategory $\calS\subseteq\T$ we write $\T/\calS$
for the corresponding \emph{quotient category}
\cite[II.2]{Ve1997}. The quotient functor has the following property.

\begin{lem}\label{le:tr}{\cite[II, Th\'eor\`eme~2.2.6]{Ve1997}}
 Let $\T$ be a triangulated category and  $\calS\subseteq\T$ a
 triangulated subcategory. The quotient functor $\T\to\T/\calS$ induces for any abelian
  category $\A$ a fully faithful functor
\[\HOM_\coh(\T/\calS,\A)\lto \HOM_\coh(\T,\A)\]
with image consisting of all functors $\T\to\A$ that annihilate $\calS$.\qed
\end{lem}

We obtain a commutative diagram
 \[
   \begin{tikzcd}
     \T\arrow[r]\arrow[d]&\T/\calS\arrow[d]\\
     \Abfree(\T)\arrow[r] &\Abfree(\T/\calS)
   \end{tikzcd}
  \]
  and claim that  the bottom horizontal functor is equivalent to a quotient
  functor.

  Let $\Abfree_\calS(\T)$ denote the Serre subcategory of
  $\Abfree(\T)$ that is generated by the image of the embedding
  $\calS\to \T\to\Abfree(\T)$.

\begin{prop}\label{pr:quot}
     Let $\T$ be a triangulated category and  $\calS\subseteq\T$ a
 triangulated subcategory. The canonical functor $\Abfree(\T)\to\Abfree(\T/\calS)$ induces an
    equivalence
   \[\Abfree(\T)/\Abfree_\calS(\T)\longiso\Abfree(\T/\calS).\] 
\end{prop}
 
\begin{proof}
For any abelian category $\A$ we obtain the following commutative diagram.
 \[
   \begin{tikzcd}
     \HOM_\ex(\Abfree(\T/\calS),\A)\arrow[r]\arrow[d,"\wr"]&\HOM_\ex(\Abfree(\T),\A)\arrow[d,"\wr"]\\
     \HOM_\coh(\T/\calS,\A)\arrow[r] &     \HOM_\coh(\T,\A)
   \end{tikzcd}
  \]
  The bottom horizontal functor is fully faithful by
  Lemma~\ref{le:tr}.  Thus the top horizontal functor is fully
  faithful, and its essential image is given by the functors that
  annihilate $\Abfree_\calS(\T)$. Then the assertion follows from
  Lemma~\ref{le:ab}.
\end{proof}

\begin{center}
  \textasteriskcentered \qquad \textasteriskcentered \qquad \textasteriskcentered
\end{center}\smallskip

Now assume that there is a universe $\frV$ strictly containing $\frU$
and fix a set $I$ in $\frV\smallsetminus\frU$. Let $\bbF$ be a field
in $\frU$ and consider the polynomial algebra
$\La=\bbF\langle I\rangle$. Modules over $\La$ are nothing but
$\bbF$-vector spaces with a distinguished set of endomorphisms indexed
by the elements of $I$. Let $\bbF$ denote the trivial $\La$-module (so
the endomorphisms given by the elements of $I$ are zero) and let $\A$
denote the category of $\La$-modules $M$ that admit a submodule
$M'\subseteq M$ such that $M'$ and $M/M'$ are finite direct sums of
copies of $\bbF$.  The category $\A$ is abelian and we write
\[\bfD^{\mathrm b}(\A)=\bfK^{\mathrm
    b}(\A)/{\mathbf{Ac}}^{\mathrm b}(\A)\] for its \emph{bounded derived
category}, which is the triangulated quotient of the homotopy category
with respect to the subcategory of acyclic complexes.  Then
$\bfK^{\mathrm b}(\A)$ is a $\frU$-category, whereas
$\bfD^{\mathrm b}(\A)$ is not. For instance, there is an isomorphism
\[\Ext^1_\La(\bbF,\bbF)\cong\Hom_\bbF(\bbF[I],\bbF)\cong \bbF^I\]
where $\bbF[I]$ denotes the free $\bbF$-module with basis $I$. This
example is a variation of \cite[Example~6.A]{Fr1964} due to
Freyd.  The calculation of $\Ext^1_\La(\bbF,\bbF)$ shows that the set
of objects of $\A$ belongs to $\frV$, and therefore
$\bfD^{\mathrm b}(\A)$ is a $\frV$-category.

\begin{prop}
  The abelian category $\Abfree(\bfK^{\mathrm b}(\A))$ is a $\frU$-category, and
  there exists a Serre subcategory such that the corresponding
  quotient is not a $\frU$-category.
\end{prop}
\begin{proof}
Let   $\Abfree_\ac(\bfK^{\mathrm b}(\A))$ denote the Serre subcategory
of   $\Abfree(\bfK^{\mathrm b}(\A))$ that is generated by all acyclic
complexes in $\bfK^{\mathrm b}(\A)$. Then Proposition~\ref{pr:quot}
yields an equivalence
\[\Abfree(\bfK^{\mathrm b}(\A))/\Abfree_\ac(\bfK^{\mathrm b}(\A))\longiso \Abfree(\bfD^{\mathrm b}(\A))\]
of $\frV$-categories. We have seen that $\bfD^{\mathrm b}(\A)$ is not a
$\frU$-category, and therefore $\Abfree(\bfD^{\mathrm b}(\A))$ is not a
$\frU$-category.
\end{proof}

\section{The abelian hull of an arbitrary category}

For a category $\C$ let $\add(\C)$ denote its \emph{additive hull},
which is obtained by taking formal finite direct sums of the objects of
the $\bbZ$-linearisation  $\bbZ\C$. The canonical functor
$\C\to\add(\C)$ induces for any additive category $\A$ an equivalence
\[\HOM_{\add}(\add(\C),\A)\longiso\HOM(\C,\A),\]
where $\HOM_{\add}(-,-)$ denotes the full subcategory of additive
functors.

For an additive category $\C$ let $\mod(\C)$ denote its completion
under finite colimits, which is obtained by adding formal cokernels of
the morphisms in $\C$, for instance by using the category of coherent
functors $\C^\op\to\Ab$ into the category of abelian groups
\cite{Au1966}. The canonical fully faithful functor $\C\to\mod(\C)$
induces for any additive category $\A$ with cokernels an equivalence
\[\HOM_{\mod}(\mod(\C),\A)\longiso\HOM_{\add}(\C,\A),\]
where $\HOM_{\mod}(-,-)$ denotes the full subcategory of additive
functors that are right exact.

The \emph{abelian hull} of an additive category $\C$ is by
definition \[\Abfree(\C):=\mod((\mod(\C)^\op))^\op.\] This
construction goes back to Freyd and Gruson \cite{Fr1966, Gr1975}. The
canonical inclusion $\C\to\Abfree(\C)$ induces for any abelian
category $\A$ an equivalence
\[\HOM_\ex(\Abfree(\C),\A)\longiso\HOM_{\add}(\C,\A);\]
see \cite[Proposition~12.4.1]{Kr2022}. We summarise these
constructions and obtain the following.

\begin{lem}\label{le:ab-hull}\pushQED{\qed}
For a category $\C$ and any abelian category $\A$ the canonical functor
$\C\to\Abfree(\add(\C))$ 
induces an equivalence
\[\HOM_\ex(\Abfree(\add(\C)),\A)\longiso\HOM(\C,\A).\qedhere\]
\end{lem}

Now consider a localisation  $\C\to \C[\Si^{-1}]$.
This yields a commutative diagram
 \[
   \begin{tikzcd}
     \C\arrow{r}\arrow[d]&\C\left[\Si^{-1}\right]\arrow{d}\\
     \Abfree(\add(\C))\arrow{r} &\Abfree(\add(\C\left[\Si^{-1}\right]))
   \end{tikzcd}
  \]
  and we claim that the bottom horizontal functor is equivalent to a
  quotient functor.

\begin{prop}\label{pr:abfree}
  Let $\C$ be a category and $\Si$ a class of morphisms in $\C$. Then
  the localisation $\C\to \C[\Si^{-1}]$ induces an
  equivalence
\[\Abfree(\add(\C))/\calS\longiso \Abfree(\add(\C[\Si^{-1}])),\]
where $\calS$ denotes the Serre subcategory generated by the kernels and
cokernels of the morphisms in $\Si$, viewed as morphisms in $\Abfree(\add(\C))$.
\end{prop}
\begin{proof}
For any abelian category $\A$ we obtain the following commutative diagram.
 \[
   \begin{tikzcd}
     \HOM_\ex(\Abfree(\add(\C[\Si^{-1}])),\A)\arrow[r]\arrow[d,"\wr"]&\HOM_\ex(\Abfree(\add(\C)),\A)\arrow[d,"\wr"]\\
     \HOM(\C[\Si^{-1}],\A)\arrow[r] &     \HOM(\C,\A)
   \end{tikzcd}
 \]
 The vertical functors are equivalences by Lemma~\ref{le:ab-hull}.
 The bottom horizontal functor is fully faithful by the universal
 property of the localisation $\C\to\C[\Si^{-1}]$, and therefore the top
 horizontal functor is fully faithful. Thus the assertion follows from
 Lemma~\ref{le:ab}.
\end{proof}

\begin{center}
  \textasteriskcentered \qquad \textasteriskcentered \qquad \textasteriskcentered
\end{center}\smallskip

Fix again a universe $\frV$ strictly containing $\frU$ and a set $I$
in $\frV\smallsetminus\frU$. We consider the following quiver
 \[
   \begin{tikzcd}
 \Gamma\colon&
 x&y_i\arrow[l,swap,"\s_i"]\arrow[r,"\t_i"]&z\;\;\; (i\in I)
   \end{tikzcd}
 \]
 and let $\C$ denote its \emph{path category}.  Thus the objects of
 $\C$ are the vertices of $\Gamma$ and the morphisms are given by the
 paths in $\Gamma$, including the trivial path at each vertex. Let
 $\Si=\{\s_i\mid i\in I\}$. Then we have in
 $\C[\Si^{-1}]$ \[\Hom(x,z)=\{\t_i\s_i^{-1}\mid i\in I\}.\]

\begin{prop}
  The abelian category $\Abfree(\add(\C))$ is a $\frU$-category, and
  there exists a Serre subcategory such that the corresponding
  quotient is not a $\frU$-category.
\end{prop}
\begin{proof}
Let   $\calS$ denote the Serre subcategory
of   $\Abfree(\add(\C))$ that is generated by the kernels and
cokernels of the morphisms in $\Si$, viewed as morphisms in $\Abfree(\add(\C))$.
Then Proposition~\ref{pr:abfree}
yields an equivalence
\[\Abfree(\add(\C))/\calS\longiso \Abfree(\add(\C[\Si^{-1}]))\]
of $\frV$-categories. We have seen that $\C[\Si^{-1}]$ is not a
$\frU$-category, and therefore $ \Abfree(\add(\C[\Si^{-1}]))$ is not a
$\frU$-category.
\end{proof}

\subsection*{Acknowledgement}

This note originates from a question in my course on representation
theory of algebras in the winter semester 2024/25.  I am grateful for
ths inspiration, and also to Greg Stevenson for his comments on this
work, in particular for suggesting the quiver in the second part of
this note.  This work was supported by the Deutsche
Forschungsgemeinschaft (SFB-TRR 358/1 2023 - 491392403).


\begin{thebibliography}{99}

\bibitem{Au1966} M. Auslander, Coherent functors, in {\it
    Proc. Conf. Categorical Algebra (La Jolla, Calif., 1965)},
  189--231, Springer-Verlag, New York, 1966.  

\bibitem{Fr1964} P.~J. Freyd, {\it Abelian categories. An introduction to
  the theory of functors}, Harper's Series in Modern Mathematics,
Harper \& Row, Publishers, New York, 1964.

\bibitem{Fr1966} P.~J. Freyd, Representations in abelian categories,
  in {\it Proc. Conf. Categorical Algebra (La Jolla, Calif., 1965)},
  95--120, Springer-Verlag, New York, 1966.
  
\bibitem{Ga1962} P. Gabriel, Des cat\'egories ab\'eliennes,
  Bull. Soc. Math. France {\bf 90} (1962), 323--448.

\bibitem{GZ1967} P. Gabriel\ and\ M. Zisman, {\it Calculus of
    fractions and homotopy theory}, Ergebnisse der Mathematik und
  ihrer Grenzgebiete, Band 35, Springer-Verlag New York, Inc., New
  York, 1967.

\bibitem{Gr1975} L. Gruson, Simple coherent functors, in {\it
    Representations of Algebras}, Springer Lecture Notes, Vol. 488,
  pp. 156--159, Springer Verlag, New York, 1975.
  
\bibitem{Kr2022} H. Krause, {\it Homological theory of
    representations}, Cambridge Studies in Advanced Mathematics, 195,
  Cambridge University Press, Cambridge, 2022.
  
\bibitem{Ve1997}J.-L. Verdier, Des cat\'egories d\'eriv\'ees des
  cat\'egories ab\'eliennes, Ast\'erisque No. 239 (1996), {\rm
    xii}+253 pp. (1997).
\end{thebibliography}
\end{document}